%% file: main.tex
\documentclass[journal]{IEEEtran}
\usepackage{color}
\usepackage[pdftex]{graphicx}
\usepackage{framed}
\usepackage{bm}
\usepackage{comment}
\usepackage[verbose]{cite}
\usepackage{algorithm}
\usepackage{algorithmic}
\newtheorem{theorem}{Theorem}
\newtheorem{assumption}{Assumption}
\newtheorem{definition}{Definition}
\newtheorem{remark}{Remark}
\newtheorem{lemma}{Lemma}
\newtheorem{corollary}{Corollary}

\newtheorem{proof}{Proof}

\usepackage{latexsym}
\def\qed{\hfill $\Box$} 

\usepackage{ascmac}

\usepackage{subfiles}
\usepackage{physics}
\usepackage{here}
\usepackage{autonum}
\usepackage{amsmath}
\usepackage{amssymb}
\usepackage{amsfonts}
\newcommand{\R}{\mathbb{R}}

\newcommand{\C}{\mathbb{C}}
\newcommand{\minimize}{\mathop{\text{minimize }}\limits}

\newcommand{\subjectto}{\mathop{\text{subject to }}}
\newcommand{\vectorize}{\operatorname{vec}}
\newcommand{\lmin}{\lambda_{\min}}
\newcommand{\lmax}{\lambda_{\max}}
\def\q<#1>{\langle #1 \rangle}
\newcommand{\proj}{\operatorname{proj}}

\newcommand{\ones}{\operatorname{ones}}
\newcommand{\rand}{\operatorname{rand}}
\newcommand{\randn}{\operatorname{randn}}
\newcommand{\normF}[1]{\norm{#1}_\mathrm{F}}
\renewcommand\Re{\operatorname{Re}}
\renewcommand\Im{\operatorname{Im}}

\ifCLASSINFOpdf
\else
\fi
\usepackage{amsmath}
\begin{document}
%
\title{Structured Output Feedback Control for Linear Quadratic Regulator Using Policy Gradient Method}
%
%
%

\author{Shokichi~Takakura and Kazuhiro~Sato
\thanks{S. Takakura is with the Department of Mathematical Informatics, 
Graduate School of Information Science and Technology, The University of Tokyo, Tokyo 113-8656, Japan,
email: masayoshi361@g.ecc.u-tokyo.ac.jp}
\thanks{K. Sato is with the Department of Mathematical Informatics, 
Graduate School of Information Science and Technology, The University of Tokyo, Tokyo 113-8656, Japan,
email: kazuhiro@mist.i.u-tokyo.ac.jp}

}

\maketitle
\thispagestyle{empty}
\pagestyle{empty}

\begin{abstract}
We consider the static output feedback control for Linear Quadratic Regulator problems with structured constraints
under the assumption that system parameters are unknown.
To solve the problem in the model free setting, we propose the policy gradient algorithm based on the gradient projection method
and show its global convergence to $\varepsilon$-stationary points.
In addition, we introduce a variance reduction technique
and show both theoretically and numerically
that it significantly reduces the variance in the gradient estimation.
We also show in the numerical experiments that the model free approach efficiently solves the problem.
\end{abstract}


%
\IEEEpeerreviewmaketitle

\subfile{introduction.tex}
\subfile{problem_formulation.tex}
\subfile{properties.tex}
\subfile{model_free.tex}
\subfile{experiments.tex}
\subfile{conclusion.tex}

\section*{Acknowledgment}
This work was supported by Japan
Society for the Promotion of Science KAKENHI under Grant 20K14760.

\appendix
\subfile{appendix.tex}

\bibliographystyle{IEEEtran}
\bibliography{references}

\vspace{-2em}

\end{document}

%% file: introduction.tex
\section{Introduction}\label{sec:intro}



Linear Quadratic Regulator (LQR), which is a well-studied framework in the optimal control theory, has been revisited from the Reinforcement Learning (RL) perspective.
For policy gradient methods, the global linear convergence to the global optima was obtained in~\cite{fazel2018global, mohammadi2021convergence}.
To obtain structured policy, Structured Policy Iteration of state feedback gains for LQR problems with a regularization term was proposed in~\cite{park2020structured}
and the local linear convergence to a stationary point was provided.
In addition, the projected gradient method for model-free state feedback LQR problems with convex constraints were studied in~\cite{hambly2021policy}. For the model based setting, the projected gradient method was studied in~\cite{bu2019lqr} and~\cite{talebi2022policy} considered linearly constrained problem for state feedback LQR problems.


However, it is difficult to observe the entire state.
That is, only some outputs are available in practice.
The static output feedback control is a practical approach to deal with such situations.
For model based control design, some iterative methods are found in~\cite{makila1987computational}
and recently, the global convergence of the gradient descent for output feedback LQR problems was shown in~\cite{fatkhullin2021optimizing} using smoothness and Lipschitz continuity on the sublevel sets of the LQR objective function.
A model free algorithm was also proposed in~\cite{zhu2015adaptive} based on integral RL.
However, policy gradient methods for static output feedback problems in the model free setting have not been well studied.

In this study, we consider a policy gradient method for the LQR problem with structured constraints
for the static output feedback control under the assumption that system parameters are unknown,
in contrast to many existing works~\cite{fazel2018global, mohammadi2021convergence, hambly2021policy}, which studied the policy gradient method for state feedback LQR problems and~\cite{fatkhullin2021optimizing}, which studied gradient methods in the model based setting.
The structured constraints are naturally introduced due to the system structure such as linear port-Hamiltonian systems~\cite{Jacob2012linear}.

{\it Our contribution:}
The main contributions of this paper are summarized as follows:
\begin{itemize}
    \item To solve the LQR problem with structured constraints in the model free setting, we propose a policy gradient projection algorithm with a gradient estimation procedure.
    \item We show the global convergence to $\varepsilon$-stationary points of our proposed algorithm using the LQR objective function properties such as bounded sublevel sets, $L$-smoothness on sublevel sets, and dependency on horizon time.
    In addition, we show that the feedback gain obtained by the proposed method asymptotically stabilizes the closed-loop system.
    We also provide the sample complexity of the gradient estimation procedure.
    \item We propose a variance reduction method using the baseline technique and show its suboptimality.
\end{itemize}

{\it Paper organization:}
In Section \ref{sec:problem}, we introduce the LQR problem for output feedback control with structured constraints.
In Section \ref{sec:properties}, we show some properties of the objective function on sublevel sets.
In Section \ref{sec:model-free}, we propose the gradient estimation method and the policy gradient projection algorithm in the model free setting.
We then show that the algorithm outputs an $\varepsilon$-stationary point with high probability.
In addition, we provide a variance reduction method and show its asymptotic optimality.
In Section \ref{sec:experiments}, we conduct some numerical experiments and show properties of our proposed method.
Section \ref{sec:conclusion} is devoted to conclusion.

{\it Notation:}
For a vector $v \in \C^n$, $v^\top$ and $v^*$ denote the transpose and conjugate transpose of $v$, respectively.
The symbol $I$ and $O$ denote the identity matrix and the zero matrix, respectively.
The symbol $\mathbb{S}^n$ denotes the set of $n\times n$ symmetric matrices.
For a matrix $A\in \R^{n\times m}$, $\normF{A}$ and $\norm{A}_2$ represent Frobenius and spectral norm of $A$, respectively,
$\lambda_i(A)$ denotes the $i$-th eigenvalue of $A$ indexed as $\Re(\lambda_1(A)) \leq \dots \leq \Re(\lambda_n(A))$,
and $\vectorize(A) \in \R^{nm}$ denotes the vectorized form of $A$.
For matrices $A, B \in \R^{n\times m}$, the inner product $\q<A, B>$ is defined as $\q<A, B> = \tr(AB^\top)$
and $A\circ B$ denotes the Hadamard product of $A$ and $B$.
Given a symmetric matrix $S \in \mathbb{S}^n$, $\lmin(S)(\lmax(S))$ denotes the minimum (maximum) eigenvalue of $S$.
Given a random variable $X$ which follows the distribution $\mathcal D$,
$E_{X \sim \mathcal{D}}[X]$ or just $E[X]$ denotes the expectation over $X \sim \mathcal D$
and $V[X]$ denotes the variance of $X$. For $z\in \C$, $\Re(z)$ ($\Im(z)$) denotes the real (imaginary) part of $z$.

%% file: problem_formulation.tex
\section{Problem Formulation}\label{sec:problem}
We consider the linear time-invariant (LTI) system
\begin{align}
  \begin{aligned}
    \dot x(t) = Ax(t)+Bu(t), \quad
    y(t)     = C x(t),\quad
    x(0) \sim \mathcal D,
  \end{aligned}\label{eq:system}
\end{align}
where $x(t) \in \R^n$ is state, $u(t)\in \R^m$ is input,
$y(t)\in \R^p$ is output, $A\in \R^{n\times n}$,
$B\in \R^{n\times m}$, and $C \in \R^{p\times n}$ are constant matrices, and $\mathcal D$ is a probability distribution over $\R^n$.
In this paper, we assume that $B$ and $C$ are not zero matrices,
and
$(A, B, C, \mathcal {D})$
is unknown unlike the situation in \cite{fatkhullin2021optimizing}.
The infinite-horizon continuous-time LQR problem is formulated as
\begin{align}
  \minimize  & E_{x(0)\sim \mathcal{D}}\qty[\int_0^\infty \qty(y^\top(t) Q y(t) + u^\top(t) R u(t))dt ] \label{eq:objectivefunction} \\
  \subjectto & ~\eqref{eq:system}
\end{align}
with constant positive definite matrices $Q \in \R^{p\times p}$ and $R\in \R^{m\times m}$.
The expectation is taken with respect to the initial state $x(0) \sim \mathcal{D}$.
For the static output feedback $u(t) = -Ky(t)$ with $K\in \R^{m\times p}$ to system~\eqref{eq:system},
the objective function~\eqref{eq:objectivefunction} becomes
$
  f(K) := E_{x(0)\sim \mathcal{D}}\qty[\tilde{f}(K;x(0))],
$
where 
\begin{align}
  \tilde f(K;v) & := \int_0^\infty \qty[v^* e^{A_K^\top t}C^\top(Q+ K^\top RK)Ce^{A_K t}v]dt\label{eq:cost}
\end{align}
for $v\in \C^n$.
Then, the closed-loop is given by
\begin{align}
    \dot x(t)  = A_Kx(t),\quad
    y(t)          = Cx(t), \label{eq:closedloop}
\end{align}
where 
\begin{align}
  A_K := A-BKC. \label{eq:AK}
\end{align}

In this paper, we consider the constraints $K\in \Omega$,
where $\Omega \subset \R^{m\times p}$ is a closed convex set
that specifies the structural information of feedback gains.
This is because a structured policy is often used in practical situations.
For example,
\begin{itemize}
    \item Decentralized control: In decentralized control, some components of $K$ need to be $0$~\cite{jovanovic2016controller}.
This implies that $\Omega$ should be a certain linear subspace of $\R^{m\times p}$.

\item Linear port-Hamiltonian system: For a linear port-Hamiltonian system~\cite{Jacob2012linear}, 
if the feedback gain is positive semi-definite, the closed loop system is also a port-Hamiltonian system and passive. To ensure passivity, $\Omega$ should be defined as the set of positive semi-definite matrices, which is closed and convex.

\end{itemize}


By using Bellman lemma~\cite{bellman1957notes}, the problem~\eqref{eq:objectivefunction} with structured constraints can be formulated as
\begin{align}
  \begin{aligned}
    \minimize_K & f(K) = \tr(X\Sigma) \\
    \subjectto  & K \in \Omega\,\, \text{and}\,\, A_K\text{ is \textit{Hurwitz}},
  \end{aligned}\label{eq:problem}
\end{align}
where
$
  \Sigma  := E[x(0)x^\top(0)]
$
and $X$ is the solution to
\begin{align}
  A_K^\top X + XA_K + C^\top \qty(K^\top RK + Q)C = 0.
\end{align}
It is difficult to solve \eqref{eq:problem}, since  $f(K)$ is non-convex and saddle points may exist~\cite{fatkhullin2021optimizing}.
Moreover,
 the feasible set may have exponentilally many disconnected components~\cite{feng2019exponential}. 
Although an iterative method was proposed in \cite{zhu2015adaptive} to obtain a suboptimal static output feedback gain in the model free setting,
it cannot be applied directly to problem \eqref{eq:problem} due to the constraint $K\in \Omega$.


To develop a model free algorithm with theoretical guarantees for solving problem \eqref{eq:problem}, we impose the following throughout this paper:
\begin{assumption}\label{assume:sigma-hurwitz}
  \indent
  \begin{enumerate}
    \item $\Sigma \succ 0$.
    \item The pair $(A, C)$ is observable.
    \item There exists $K_0 \in \Omega$
          such that $A_{K_0}$ is \textit{Hurwitz}
          and $K_0$ is known.
  \end{enumerate}
\end{assumption}

Since $A_{K_0}$ is \textit{Hurwitz},
there exist positive definite matrices $G, H$ and a skew-adjoint matrix $J$
such that
$
  A_{K_0} = (J-G)H.
$
The proof is found in~\cite{prajna2002lmi}.
Let $H = L^\top L$ be the Cholesky decomposition.
Using the coordinate transformation $x'(t) = Lx(t)$, the closed-loop system~\eqref{eq:closedloop} becomes
\begin{align}
  \dot{x'}(t)  = A'_{K_0}x'(t), \quad 
  y(t)         = C'x'(t),
\end{align}
where $A'_{K_0}  = LJL^\top-LGL^\top, C' = CL^{-1}$.
Since $LGL^\top\succ 0$ and $LJL^\top = -(LJL^\top)^\top$, we have
  $A'_{K_0}+{A'_{K_0}}^\top = -2LGL^\top \prec 0$.
In the following, we assume system \eqref{eq:closedloop} after the above coordinate transformation, because we consider a static output feedback that is invariant by the coordinate transformation. That is, without loss of generality, we can assume
$A_{K_0}+A_{K_0}^\top \prec 0$.

Under Assumption \ref{assume:sigma-hurwitz},
 $f(K)$ of \eqref{eq:problem} is defined only on the set $S$ of stabilizing controllers, which is defined as
\begin{align}
  S = \{K\in \R^{m\times p}\mid A_K\text{ is \textit{Hurwitz}}\}. \label{eq:S}
\end{align}
If $K\notin S$, there exists an eigenvalue $\mu$ of $A$ such that ${\rm Re}(\mu) \geq 0$ and $f(K)$ goes to infinity.

\begin{remark} \label{remark1}
The objective function of
 problem \eqref{eq:objectivefunction}
is not a
standard LQR cost
as in some previous researches~\cite{modares2016optimal, rizvi2018output}. While similar convergence properties to the standard LQR cost can be obtained for our formulation in the model based setting if $(A, C)$ is observable~\cite{fatkhullin2021optimizing}, more detailed studies of the objective function properties are necessary for model-free version of the convergence analysis.  
\end{remark}



%% file: properties.tex
\section{Properties of the Objective Function}\label{sec:properties}
In this section, we prove some properties of the objective function $f(K)$ in~\eqref{eq:problem}
for the convergence analysis of the gradient method presented in Section~\ref{sec:model-free}.

\subsection{Norm bounds}
In this section, we show some matrix norm inequalities.

We define the sublevel set by
$
  S(a) = \{K \in S \mid f(K)\leq a\},
$
where $S$ is defined as \eqref{eq:S}. Thus, all elements in $S(a)$ are stabilizing feedback gains.

Using the same argument as Lemma C.2 in~\cite{fatkhullin2021optimizing}, we have $\norm{K}\leq \kappa(a)$ for $K\in S(a)$, where
\begin{align}
    \kappa(a)&:=\frac{2\norm{B}_2\norm{C}_2a}{\lmin{(\Sigma)}\lmin{(R)}\lmin{(CC^\top})}+\frac{\norm{A}_2}{\norm{B}_2\norm{C}_2}.
\end{align}

Next, for $K \in S(a)$,
we provide an upper bound on the solutions to the following Lyapunov equations:
\begin{align}
  A_K^\top X + X A_K + C^\top(Q+K^\top R K)C   & = 0, \label{eq:lyapunovX}     \\
  A_K Y + Y A_K^\top + \Sigma                  & = 0, \label{eq:lyapunovY}     \\
  A_K Y' + Y'A_K^\top - \qty(BECY+(BECY)^\top) & = 0,\label{eq:lyapunovYprime}
\end{align}
where $A_K$ is defined in~\eqref{eq:AK} and $E \in \R^{m\times p}$ is a given matrix.
Note that $X$, $Y$, and $Y'$ uniquely exist, because
$K \in S(a)$ implies that
$A_K$ is \textit{Hurwitz}~\cite{bellman1957notes}.
To simplify the notation,
using $\sigma := -\frac{1}{2}\lmax(A_{K_0}+A_{K_0}^\top)>0$,
we define
\begin{align}
  \xi & := \frac{1}{4\norm{B}_2\kappa(a)},\,\, \mathfrak{X}(a) :=\frac{a}{\lmin(\Sigma)},   \label{eq:frakX} \\
  \mathfrak{Y}(a)  & :=\max\qty(\frac{a}{\xi^2\lmin(Q)}, \frac{\norm{\Sigma}}{\sigma}),                               \label{eq:frakY}\\
  \mathfrak{Y}'(a) &:=2\frac{\norm{B}_2\norm{C}_2\mathfrak{Y}(a)^2}{\lmin(\Sigma)}\label{eq:frakYprime}
\end{align}
where $K_0$ satisfies 3 in Assumption \ref{assume:sigma-hurwitz}.

\begin{lemma}\label{lem:xybound}
  Let $X, Y, Y'$ be the solution to~\eqref{eq:lyapunovX},~\eqref{eq:lyapunovY}, and~\eqref{eq:lyapunovYprime}, respectively.
  Assume that $\normF{E} = 1$.
  Then, for any $K\in S(a)$,
    $\norm{X}_2  \leq \mathfrak{X}(a)$,\,\,
    $\norm{Y}_2  \leq \mathfrak{Y}(a)$,\,\, 
    $\norm{Y'}_2 \leq \mathfrak{Y}'(a)$.
\end{lemma}

\begin{proof}
See Appendix \ref{sec:proof-lemma1}. \qed
\end{proof}

\subsection{$L$-smoothness of $f(K)$ }

A differentiable function is called $L$-smooth if its gradient is $L$-Lipschitz continuous.
For our objective function, we have the following result.
\begin{theorem}
  For any $a\in \R$, $f(K)$ in~\eqref{eq:problem} is $L$-smooth on $S(a)$ with the constant 
  \begin{align}
    L & := 2\lmax(R)\normF{C}^2\mathfrak{Y}(a) \label{eq:L}\\
    &+4\qty(\sqrt{n}\normF{R}\kappa(a)\normF{C}+n\normF{B} \mathfrak{X}(a)) \mathfrak{Y}'(a) \normF{C},\nonumber
  \end{align}
  where $\mathfrak{Y}(a)$ and $\mathfrak{Y}'(a)$ are defined in~\eqref{eq:frakY} and~\eqref{eq:frakYprime}. 
\end{theorem}
\begin{proof}
Theorem 3.15 in~\cite{fatkhullin2021optimizing} cannot be applied to our setting directly, because $\lambda_1(C^\top QC)$ may be $0$.
However, by replacing the norm bounds in the proof of Theorem 3.15 with those in Lemma 1, we obtain the result. \qed
\end{proof}


%% file: model_free.tex
\section{Model Free Algorithm} \label{sec:model-free}
In this section, we consider problem \eqref{eq:problem} in the model free setting. That is, we assume that $(A, B, C, \mathcal D)$ in system \eqref{eq:system} is unknown. First, we introduce a gradient estimation algorithm
based on the derivative-free optimization and show that the estimated gradient is close to the exact gradient with high probability.
Then, we provide Policy Gradient Projection Algorithm in Algorithm \ref{alg:mfgp}.
Despite the error in the estimated gradient, we can prove the global convergence to $\varepsilon$-stationary points. 
Finally, we propose a variance reduction method using the baseline technique and prove its optimality.
\subsection{Gradient estimation}
In the model free setting, the exact gradient of the objective function $f(K)$
in~\eqref{eq:objectivefunction} cannot be accessed directly.
Then, based on the zeroth-order or derivative-free optimization approach~\cite{fazel2018global},
we propose Algorithm~\ref{alg:gradest1}
to calculate the stochastic estimate $\hat \grad f(K)$ of the gradient $\grad f(K)$.

\begin{algorithm}
  \caption{Gradient Estimation}
  \label{alg:gradest1}
  \begin{algorithmic}[1]
    \REQUIRE{$K\in \Omega, N>0, r > 0, \tau > 0$}
    \FOR{$i=1$ to $N$}
    \STATE{Sample $U_i$ from the uniform distribution $\mathcal{S}$ over matrices with $\normF{U_i} = \sqrt{mp}$.}
    \STATE{Simulate the system
      \begin{align}
        \dot x_i(t)  = Ax_i(t)+Bu_i(t),\quad
        y_i(t)       = Cx_i(t),
      \end{align}
      where $u_i(t) = -(K+rU_i)y_i(t)$, $x_i(0) \sim \mathcal D$
      until time $\tau$ and
      calculate the empirical cost 
       $c_i = \int_0^\tau \qty[y_i^\top(t) Q y_i(t) + u_i(t)^\top R u_i(t)]dt$.
     }
    \ENDFOR
    \STATE{Define the estimated gradient by
      \begin{align}
        \hat \grad f(K) = \frac{1}{rN}\sum_{i=1}^N c_iU_i. \label{eq:hatgrad}
      \end{align}}
    \RETURN $\hat \grad f(K)$.
  \end{algorithmic}
\end{algorithm}

In this section, we assume the following in addition to Assumption \ref{assume:sigma-hurwitz}.
\begin{assumption}\label{assume:gradest}
  \indent
  \begin{itemize}
    \item ${U}_i$ and $x_i(0)$ are independent.
    \item $K\in S(a)$ with a given constant $a$.
    \item The distribution $\mathcal D$ of initial state $x(0) \sim \mathcal{D}$ satisfies $\norm{x(0)}\leq P\quad \text{a.s.}$ for a constant $P > 0$.
  \end{itemize}
\end{assumption}
The second assumption is justified in Section~\ref{sec:mfgp}.

The following lemma ensures $K+rU \in S(2a)$ for sufficiently small $r$.
\begin{lemma}
  \label{lem:smallr}
  There exists $r_0  > 0$ such that
  for any $r \leq r_0$, $K\in S(a)$ and $U$ such that $\norm{U} = \sqrt{mp}$, we have $K+rU\in S(2a)$.
\end{lemma}
\begin{proof}
Using the same argument as Lemma 4 in~\cite{mohammadi2021convergence}, we obtain the result. \qed
\end{proof}

The following theorem is an extension of Lemma 27 in Supplementary material of~\cite{fazel2018global} 
to output feedback control.
To prove this, we derive some inequalities in Appendix~\ref{sec:proof-theorem2}.

\begin{theorem}\label{thm:totalerror}
Let $\hat{\nabla} f(K)$ be defined as \eqref{eq:hatgrad}.
  For any $\varepsilon' > 0$ and $\delta > 0$, set $r = O(\varepsilon'), \tau = O(\log 1/\varepsilon'), N = O((\log 1/\delta)/\varepsilon'^4)$. Then, we have
  \begin{align}
    &\normF{\hat \grad f(K)-\grad f(K)} \leq \varepsilon' \label{eq:high-probability},
  \end{align}
  for any $K\in S(a)$,
  with probability greater than $1-\delta$.
\end{theorem}
\begin{proof}
  See Appendix~\ref{sec:proof-theorem2}. \qed
\end{proof}

\subsection{Convergence properties}\label{sec:mfgp}
In this section, we prove the global convergence of the policy gradient method in the model free setting.
We show a model free control algorithm, Policy Gradient Projection, in Algorithm~\ref{alg:mfgp}.
The positive integer $T$ is the iteration number, $\alpha$ is the step size, $K_0$ is the initial point of the feedback gain $K$,
and $\proj$ is the projection onto $\Omega$ with respect to the Frobenius norm.
The termination condition $\normF{K_{i+1}-K_i}\leq \varepsilon\alpha$ is added for technical reasons.
\begin{algorithm}
  \caption{Policy Gradient Projection}
  \label{alg:mfgp}
  \begin{algorithmic}[1]
    \REQUIRE{$T > 0, \alpha > 0, \varepsilon > 0, K_0\in \Omega, N > 0, r>0, \tau>0$}
    \FOR{$i=0$ to $T-1$}
    \STATE{Calculate $\hat \grad f(K_i)$ using Algorithm \ref{alg:gradest1}.}
    \STATE{$K_{i+1} := \proj\qty(K_i-\alpha \hat \grad f(K_{i}))$.}
    \IF{$\normF{K_{i+1}-K_i}\leq \varepsilon\alpha$}
    \RETURN $K_i$
    \ENDIF
    \ENDFOR
    \RETURN $K_T$
  \end{algorithmic}
\end{algorithm}

Lemma~\ref{lem:projection} describes the property of orthogonal projections,
which is essential to our convergence analysis.
\begin{lemma}
  \label{lem:projection}
  Let $\proj:\R^{m\times p}\to \Omega$ be an orthogonal projection onto $\Omega$.
  For any $x\in \Omega$ and $y\in \R^{m\times p}$, we have
  $\q<x-\proj(y), y-\proj(y)>  \leq 0$.
\end{lemma}
\begin{proof}
  See Theorem 6.41 of \cite{beck2017first}. \qed
\end{proof}

The following definition is required to show our main result.
\begin{definition}
  \label{def:epsst}
  For positive constants $\alpha$ and $\varepsilon$,
  $K$ is called an $\varepsilon$-stationary point if $\normF{G_\alpha(K)} \leq \varepsilon$, where $K^+ := \proj(K-\alpha\grad f(K))$ and $G_\alpha(K) := \frac{1}{\alpha}(K^+-K)$.
\end{definition}

The following theorem is a main result, which is an extension of Theorem 4.2 in~\cite{fatkhullin2021optimizing} to the model-free and constrained problems. 
The proof is based on the proof for projected gradient method without gradient error for $L$-smooth functions~\cite{beck2017first}. However, in the presence of gradient errors, the termination condition and some extra arguments to bound the effect of the difference between the true and estimated gradients are required.

\begin{theorem}\label{thm:modelfreeconvergence}
  Assume that the constants $N, r, \tau$ satisfies the condition in Theorem~\ref{thm:totalerror} with $\varepsilon' = \lambda \varepsilon$
  for the given constants $0<\lambda<1, \varepsilon > 0, \delta > 0$.
  Let $\{K_i\}_{i=0}^{T'}$ be the sequence generated by the Algorithm~\ref{alg:mfgp},
  where $T'$ is the total number of iterations, which can be different from $T$ due to the terminate condition.
  For a step size $\alpha \in \qty(0, \frac{2(1-\lambda)}{L})$,
  where $L$ is the Lipschitz constant of $\grad f$ on $S_0$,
  we have the following result with probability greater than $1-T\delta$.
  \begin{itemize}
    \item The sequence $\{K_i\}_{i=0}^{T'}$ remains in $S$ and $\{f(K_i)\}_{i=0}^{T'}$ is strictly decreasing. That is, for any $0\leq i\leq T'-1$,
          \begin{align}
            f(K_{i+1}) < f(K_i).\label{eq:tmp4}
          \end{align}
    \item If $T > \frac{f(K_0)}{\varepsilon^2\alpha^2\qty(\frac{1-\lambda}{\alpha}-\frac{L}{2})}$,
          $K_{T'}$ is a $(1+\lambda)\varepsilon$-stationary point.
  \end{itemize}
\end{theorem}
\begin{proof}
  We define $G'_\alpha(K) = \frac{1}{\alpha}{v}$, where $v:=\hat{K}^+-K$ and $\hat{K}^+ := \proj(K-\alpha\hat \grad f(K))$.
  First, we show that if $\normF{G'_\alpha(K)} \leq \varepsilon$ and \eqref{eq:high-probability} holds,
  $K \in S_0$ is a $(1+\lambda)\varepsilon$-stationary point.
  \begin{align}
    \normF{G_\alpha(K)} & \leq \normF{G'_\alpha(K)}+\frac{1}{\alpha}\normF{K^+-\hat{K}^+} \\
                & \leq \varepsilon+\normF{\grad f(K)-\hat\grad f(K)}                \leq (1+\lambda)\varepsilon.
  \end{align}
  The second inequality holds, because projections onto convex sets are contractive,
  and the last inequality follows from~\eqref{eq:high-probability}.

  Next, we show $\normF{G'_\alpha(K_{T'})} \leq \varepsilon$ with high probability.
  The termination condition ensures $\normF{G'_\alpha(K_i)} > \varepsilon$ for $i = 0, \dots, T'-1$.
  Assume that $\normF{G'_\alpha(K)} > \varepsilon$ and Eq.~\eqref{eq:high-probability} with $\varepsilon' = \lambda \varepsilon$ holds for $K\in S_0$.
    We define $K_t:=K+tv$ and $t^* := \max\qty{t > 0 \mid f(K_{t'}) \leq f(K_0)\quad (0\leq \forall t' \leq t)}$.
  Then, Lemma~\ref{lem:projection} yields
  \begin{align}
    \q<\hat \grad f(K), v> & \leq -\frac{1}{\alpha}\normF{v}^2\label{eq:mf-proj}
  \end{align}
  and $L$-smoothness of $f(K)$ on $S_0$ implies
  \begin{align}
    f(K_t)-f(K) & \leq \q<\grad f(K), K_t-K>+\frac{L}{2}\normF{K_t-K}^2\\
    &=t\q<\grad f(K), v>+\frac{Lt^2}{2}\normF{v}^2.\label{eq:mf-L-smooth}
  \end{align}
  
  By adding both sides of the equations \eqref{eq:mf-proj} multiplied by $t$ and \eqref{eq:mf-L-smooth}, we obtain
  \begin{align}
    f(K_t)&-f(K)
 <\qty(\frac{Lt^2}{2}-t\frac{1-\lambda}{\alpha})\normF{v}^2.\label{eq:mf-sufficient-decrease}
  \end{align}
  where we used
\eqref{eq:high-probability} and $G'_\alpha(K)=\normF{v}/\alpha > \varepsilon$.
  For $t = t^*$, we have
  \begin{align}
      0 &= f(K_{t^*})-f(K)\leq \qty(\frac{L{t^*}^2}{2}-{t^*}\frac{1-\lambda}{\alpha})\normF{v}^2
  \end{align}
  Since $\norm{v} > 0$, we have $t^* \geq \frac{2(1-\lambda)}{L\alpha}\geq 0$.
  Therefore,~\eqref{eq:mf-sufficient-decrease} holds for $t = 1$.
  Eq.~\eqref{eq:mf-sufficient-decrease} with $K=K_i$ and $t=1$ leads us to
  \begin{align}
    f(K_{i+1})-f(K_i) & < \qty(\frac{L}{2}-\frac{1-\lambda}{\alpha})\normF{K_{i+1}-K_i}^2 <0,
  \end{align}
  because $\alpha \in (0, \frac{2(1-\lambda)}{L})$.
  Thus,~\eqref{eq:tmp4} holds for any $0\leq i\leq T'-1$.

  If $T' < T$, the termination condition ensures $\normF{G'_\alpha(K_{T'})} \leq \varepsilon$.
  Therefore, it suffices to show $T' < T$.
  Since $T' \leq T$ by definition, suppose that $T' = T$.
  Then, we obtain
  \begin{align}
    f(K_0)  \geq f(K_0)-f(K_{T})
          \geq \qty(\frac{1-\lambda}{\alpha}-\frac{L}{2})T\varepsilon^2\alpha^2.
  \end{align}
  The assumption $T > \frac{f(K_0)}{\varepsilon^2\alpha^2\qty(\frac{1-\lambda}{\alpha}-\frac{L}{2})}$
  yields $f(K_0)>f(K_0)$, which is a contradiction. Thus, $T'<T$, that is, $\normF{G'_\alpha(K_{T'})}\leq \varepsilon$.
  From Theorem~\ref{thm:totalerror}, the probability that Eq. \eqref{eq:high-probability} holds for $K = K_i(i=0, \dots, T')$ is greater than
  $(1-\delta)^{T'+1}\geq 1-T\delta$.
  This completes the proof.\qed
\end{proof}

The convergence rate $T = O(1/\varepsilon^2)$ is essentially the same as the rate of projected gradient method without gradient error for $L$-smooth functions~\cite{beck2017first}. For sample complexity, the total number of samples $TN = O(\log(1/\delta)/\varepsilon^6)$ is worse than $O(\log(1/\delta)/\varepsilon^4)$ of the zeroth-order proximal gradient descent with two points evaluation in~\cite{ghadimi_mini-batch_2016} since we cannot evaluate the two cost function values with two different feedback gains for the same initial state due to the randomness of the initial state. Note that the same rate to ours was obtained for discrete-time state feedback LQR problems in the model free setting~\cite{hambly2021policy} but not known for the model free and output feedback setting.

\subsection{Variance reduction}
Policy gradient methods tend to suffer from a large variance, which leads to slow learning~\cite{grondman2012survey}.
The use of baseline is one of the variance reduction techniques for policy gradient methods~\cite{grondman2012survey}.
State-depending functions are often used as a baseline, because it does not add any bias to the estimated gradient~\cite{degris2012model-free, mnih2016asynchronous}.
In this section, we propose to use the finite horizon cost function as a baseline and show its optimality.

For a baseline function $b(x)$, the estimated gradient $\bar \grad f$ is defined as
\begin{align}
    \bar \grad f(K)  &:= \frac{1}{rN}\sum_{i=1}^N (\tilde f_\tau(K+rU_i;x_i(0))-b(x_i(0)))U_i,\label{eq:grad-baseline-b}
\end{align}
where the finite horizon cost function is defined as
\begin{align}
    \tilde f_\tau(K;x(0)) := \int_0^\tau y^\top(t) ( Q + K^\top RK)y(t)dt\label{eq:ftau},
\end{align}
which satisfies $\lim_{\tau \to \infty} \tilde f_\tau(K;x(0)) = \tilde f(K;x(0))$.
Because 
\begin{align}
    \hat \grad f(K) = \frac{1}{rN}\sum_{i=1}^N \tilde f_{\tau}(K+r U_i;x_i(0))U_i,
\end{align}
the estimated gradient $\bar \grad f(K)$ satisfies
\begin{align}
  E\qty[\bar \grad f(K)] & = E_{x_i(0)\sim \mathcal{D}, U_i \sim \mathcal{S}}\qty[\hat \grad f(K)] \\
  & \quad - \frac{1}{rN}\sum_{i=1}^N E_{x_i(0)\sim \mathcal{D}, U_i \sim \mathcal{S}}\qty[b(x_i(0))U_i] \\
                         & = E_{x_i(0)\sim \mathcal{D}, U_i\sim \mathcal{S}}\qty[\hat \grad f(K)].
\end{align}
The second equality holds from the assumption that $U_i$ and $x_i(0)$ are independent, and
$E[U_i] = 0$.
Thus, the bias in $\bar \grad f$ is the same as the one in $\hat \grad f$.

In terms of the variance, the baseline $b(x(0)) = \tilde f_\tau(K;x(0))$ is almost optimal for small $r$.
\begin{theorem}\label{thm:baseline}
  For $a > 0$, $r \leq r_0$, $\tau > 0$, and $K\in S(a)$,
  the optimal baseline $b^*(x(0))$ which minimizes the variance of the estimated gradient~\eqref{eq:grad-baseline-b} is given by
  \begin{align}
    b^*(x(0)) = E_{U\sim \mathcal{S}}\qty[\tilde f_\tau(K+rU;x(0))],\label{eq:b-star}
  \end{align}
  where $\mathcal{S}$ is defined in Algorithm~\ref{alg:gradest1}
  and
  \begin{align}
    \lim_{r\to 0}b^*(x(0)) = \tilde f_\tau(K;x(0)).\label{eq:limit}
  \end{align}
\end{theorem}
\begin{proof}
  Since $\qty{x_i(0)}_{i=1}^N, \qty{U_i}_{i=1}^N$ are assumed to be independent, $\qty{f_\tau(K;x_i(0))-b(x_i(0)))U_i}_{i=1}^N$ are independent and we have
  \begin{align}
    &V\qty[\bar \grad f(K)] \label{eq:baseline-variance}\\
                           &= \frac{1}{r^2N}\left[E_{x(0)\sim \mathcal{D}, U\sim \mathcal{S} }\qty[\normF{(\tilde f_\tau(K+rU;x(0))-b(x(0)))U}^2]\right.        \\
                           & \quad\quad-\left.\normF{E_{x(0)\sim \mathcal{D}, U\sim \mathcal{S} }[(\tilde f_\tau(K+rU;x(0))-b(x(0)))U]}^2\right]\nonumber
  \end{align}
  for any baseline $b(x)$.
  Because we have assumed that $x(0)$ and $U$ are independent and $E_{U\sim \mathcal{S}}[U] = 0$,
  \begin{align}
      E_{x(0)\sim \mathcal{D}, U\sim \mathcal{S}}\qty[b(x(0))U] = E_{x(0) \sim \mathcal{D}}\qty[b(x(0))]E_{U\sim \mathcal{S}}[U] = 0.
  \end{align}
  Thus, the expectation $E[(\tilde{f}_\tau(K+rU;x(0))-b(x(0)))U]$ is independent of the choice of a baseline $b(x)$,
  that is,
  \begin{align}
& E_{x(0)\sim \mathcal{D}, U \sim \mathcal{S}}\qty[(\tilde f_\tau(K+rU;x(0))-b(x(0)))U] \\ =& E_{x(0)\sim \mathcal{D}, U \sim \mathcal{S}}\qty[\tilde f_\tau(K+rU;x(0))U].
  \end{align}
  Then, the second term in \eqref{eq:baseline-variance} is independent of the choice of $b(x)$ and we just need to minimize the first term in \eqref{eq:baseline-variance},
  \begin{align}
    & E_{x(0) \sim \mathcal{D}, U\sim \mathcal{S}}\qty[\norm{\qty(\tilde f_\tau(K+rU;x(0))-b(x(0)))U}^2_F] \\
     =& mp E_{x(0)\sim \mathcal{D}}\qty[E_{U\sim S}\qty[\qty(\tilde f_\tau(K+rU;x(0))-b(x(0)))^2]].
  \end{align}
  Since the expectation $E_{U\sim \mathcal{S}}\qty[\tilde f_\tau(K+rU;x(0))]$ minimizes the mean squared error $E_{U\sim S}\qty[\qty(\tilde f_\tau(K+rU;x(0))-b(x(0)))^2]$ for any $x(0)$, 
  the optimal baseline is given by \eqref{eq:b-star}.
  Eq.~\eqref{eq:limit} follows from the continuity of $\tilde f_\tau(K+rU;x(0))$.\qed
\end{proof}

Based on Theorem~\ref{thm:baseline}, we propose to use $\tilde f_\tau(K;x(0))$ as a baseline.
In the model free setting, $\tilde f_\tau(K;x(0))$ cannot be computed directly for a given $x(0)$ in the same manner as $c_i$ because we cannot specify the initial state in the estimation procedure of $c_i$.
Therefore, we provide the estimation procedure for $\tilde f_\tau(K;x(0))$ in Algorithm~\ref{alg:estimate-f}. In the following, we define $\bar{y}(t;x(0)) = [y(t)^\top, y(t+h_1)^\top, \dots, y(t+h_{D-1})^\top]^\top \in \R^{pD}(0\leq h_1<h_2<\dots<h_{D-1}=T)$.
\begin{algorithm}
  \caption{Estimate $\tilde f_{\tau}(K;x(0))$}
  \label{alg:estimate-f}
  \begin{algorithmic}[1]
    \REQUIRE{$K\in S, x(0) \in \R^n, s > 0$.}
    \FOR{$i=1$ to $n(n+1)/2$}
    \STATE{Simulate system~\eqref{eq:system} for $x(0) = x_i(0) \sim \mathcal D$ until time $s$.}
    \ENDFOR
    \STATE{Solve the following equations for $\hat P(X)$.
    \begin{align} 
        & \bar y(0;x_i(0))^\top \hat P(K)\bar y(0;x_i(0)) \\
        & = \tilde f_s(K;x_i(0)) + \bar y(s;x_i(0))^\top \hat P(K)\bar y(s;x_i(0)),\label{eq:bellmaneq}
    \end{align}}
    where $\tilde f_s$ is defined as \eqref{eq:ftau}.
    \STATE{Define $\hat f_\tau(K;x(0)) = \bar y(0;x(0))^\top \hat P(K)\bar y(0;x(0)) - \bar y(\tau;x(0))^\top \hat P(K)\bar y(\tau;x(0))$.}
    \RETURN $\hat f_\tau(K;x(0))$.
  \end{algorithmic}
\end{algorithm}
The following theorem ensures that the estimated cost $\hat f_\tau(K;x(0))$ is equal to $\tilde f_\tau(K;x(0))$.
\begin{theorem}\label{thm:estimate-f}
  For any $T > 0$, $D > 2(n-1) + \frac{T}{2\pi} \beta$, where $\beta = 2(\norm{A}_2 + \norm{B}_2\norm{C}_2\kappa(a))$,
  if $x_j(0)x_j(0)^\top - x_j(s)x_j(s)^\top(j=1, \dots, \frac{n(n+1)}{2})$
  are linearly independent on $\mathbb{S}^n$ for $s$ in Algorithm~\ref{alg:estimate-f}, then $\hat f_\tau(K;x(0)) = \tilde f_\tau(K;x(0))$ for any $K \in S$ and $x(0)$.
\end{theorem}
\begin{proof}
  See appendix~\ref{proof:thm:estimate-f}.
\end{proof}

The following theorem shows that
the assumption of Theorem~\ref{thm:estimate-f} holds with probability 1.

\begin{theorem}
If the distribution $\mathcal {D}$ has a probability density function, $x_j(0)x_j(0)^\top - x_j(s)x_j(s)^\top(j=1, \dots, \frac{n(n+1)}{2})$ are linearly independent on $\mathbb{S}^n$ with probability 1.
\end{theorem}
\begin{proof}
Let $v_i = x_i(0)x_i(0)^\top - x_i(s)x_i(s)^\top$ and $V_m$ be a linear subspace generated by $\qty{v_i}_{i=1}^m$ ($1 \leq m \leq \frac{n(n+1)}{2}$).
For $m < \frac{n(n+1)}{2}$, $V_m$ is a proper subspace of $\mathbb{S}^n$ and there exists $\tilde v_m \neq O \in \mathbb{S}^n$ orthogonal to $V_m$.
Since $x_i(s) = e^{A_Ks}x_i(0)$, we have
\begin{align}
    &\langle v_{m+1}, \tilde v_m\rangle\\
    &= \langle x_{m+1}(0)x_{m+1}(0)^\top - e^{A_Ks}x_{m+1}(0)x_{m+1}(0)^\top e^{A_K^\top s}, \tilde v_m\rangle\\
    &= x_{m+1}(0)^\top \tilde v'_m x_{m+1}(0),
\end{align}
    where $\tilde v'_m := \tilde v_{m} - e^{A^\top_Ks}\tilde v_{m}e^{A_Ks}$.
    Note $\tilde v'_m \neq O$, because $A_K$ is \textit{Hurwitz} and the solution to the discrete Lyapunov equation $v - e^{A^\top_Ks} v e^{A_Ks} = 0$ is only $v = O$. Since $\tilde v'_m \in \mathbb{S}^n$, there exists an orthogonal matrix $U$ such that $U^\top \tilde v'_m U = \operatorname{diag}(\mu_1, \dots, \mu_n)$, where $\qty{\mu_i}_{i=1}^n$ are eigenvalues of $\tilde v'_m$. Without loss of generality, we assume $\mu_n \neq 0$ since $\tilde v'_m \neq O$. For $x_{m+1}(0) \sim \mathcal {D}$, let $z = Ux_{m+1}(0)$. Then, $x_{m+1}(0)^\top \tilde v'_m x_{m+1}(0) = \sum_{i=1}^n \mu_i z_i^2$. Since $\mathcal{D}$ has a probability density function, the distribution of $z$ has also a probability density function $g(z_1, \dots, z_n)$. For $z_1, \dots, z_{n-1}$ such that $\int_{z_n} g(z_1, \dots, z_{n})\mathrm{d}z_n > 0$, the conditional probability density function of $z_n$ is given by $\frac{g(z_1, \dots, z_n)}{\int_{z_n} g(z_1, \dots, z_{n})\mathrm{d}z_n}$ and the conditional probability that $z_n$ satisfies $\sum_{i=1}^n \mu_i z_i^2 = 0$ is zero, because there are at most two $z_n$ in $\R$ which satisfy $z_n^2 = \frac{\sum_{i=1}^{n-1} \mu_i z_i^2}{\mu_n}$. 
    Therefore, the probability that $\langle v_{m+1}, \tilde v_m\rangle = x_{m+1}(0)^\top \tilde v'_m x_{m+1}(0) = 0$ is zero.
    That is,
    $v_{m+1} \notin V_m$ with probability 1. 
    By induction, we obtain the result. \qed
\end{proof}

%% file: experiments.tex
\section{Numerical Experiments}\label{sec:experiments}
In this section, we numerically demonstrate that the policy gradient projection algorithm can solve
the LQR problem efficiently in the model free setting.
Based on \cite{fatkhullin2021optimizing}, we consider the problem \eqref{eq:problem} with
$A  =(J-G)H$,
$B  = \ones(10, 4)+\frac{1}{2}\rand(10, 4)$,
$C  = \ones(2, 10)+\frac{1}{2}\rand(2, 10)$,
$Q  = I$,
$R  = I$,
where
$J = \tilde J - \tilde J^\top$, $\tilde J = \randn(10, 10)$,
$G = \tilde G$  $\tilde G^\top$, $\tilde G = \randn(10, 10)$,
$H = \tilde H \tilde H^\top$, $\tilde H = \randn(10, 10)$,
$\ones(a, b)$ is $a\times b$ matrix of ones,
$\rand(a, b)$ is $a\times b$ matrix with all entries distributed as the uniform distribution on $[0, 1]$,
and $\randn(a, b)$ is $a\times b$ matrix with all entries distributed as the standard normal distribution.
We assume the distribution $\mathcal{D}$ is the uniform distribution on $[-1, 1]^n$.
Since $J$ is skew-adjoint and $G, H$ are positive definite,
$A$ is \textit{Hurwitz}, as mentioned in Section \ref{sec:problem}.
Therefore, $A_{K_0}$ is \textit{Hurwitz} for $K_0 = O$.
We set the parameters $r = 0.01, \tau = 100$
and define $\Omega$ by
\begin{align}
	\Omega = \{K\in \R^{4\times 2}\mid K\circ S = O\},\quad S =\begin{pmatrix}
		1 & 1 & 0 & 0 \\
		0 & 0 & 1 & 1
	\end{pmatrix}^\top.
\end{align}

Fig.~\ref{fig:totalerror} illustrates that
the mean and standard deviation of 20 trials of the relative error
$\frac{\normF{\bar \grad f(K)-\grad f(K)}}{\normF{\grad f(K)}}$
in gradient estimation.
The relative error with variance reduction was much smaller than that of Algorithm 1.

Fig.~\ref{fig:convergence-curve} illustrates the mean and variance of 20 trials of the convergence curve of Algorithm~\ref{alg:mfgp},
where we set $\alpha = 2\cdot 10^{-4}, 2\cdot 10^{-5}$ and $N=15$ for Algorithm ~\ref{alg:mfgp} with baseline and $N=70$ for Algorithm~\ref{alg:mfgp} without baseline. This is because the estimation procedure of the baseline requires additional $\frac{n(n+1)}{2}=55$ samples.
Because of the large variance of the estimated gradient, Algorithm~\ref{alg:mfgp} without baseline made the system unstable quickly in the case of $\alpha=2\cdot 10^{-4}$.
As shown in Fig.~\ref{fig:convergence-curve}, the convergence rate of Algorithm~\ref{alg:mfgp} is sublinear and Algorithm~\ref{alg:mfgp} with baseline is more stable than Algorithm~\ref{alg:mfgp} without baseline even if we take into account additional 55 samples required to estimate the baseline.

\begin{figure}[t]
	\centering
	\begin{minipage}{0.48\columnwidth}
		\centering
		\includegraphics[width=1\columnwidth]{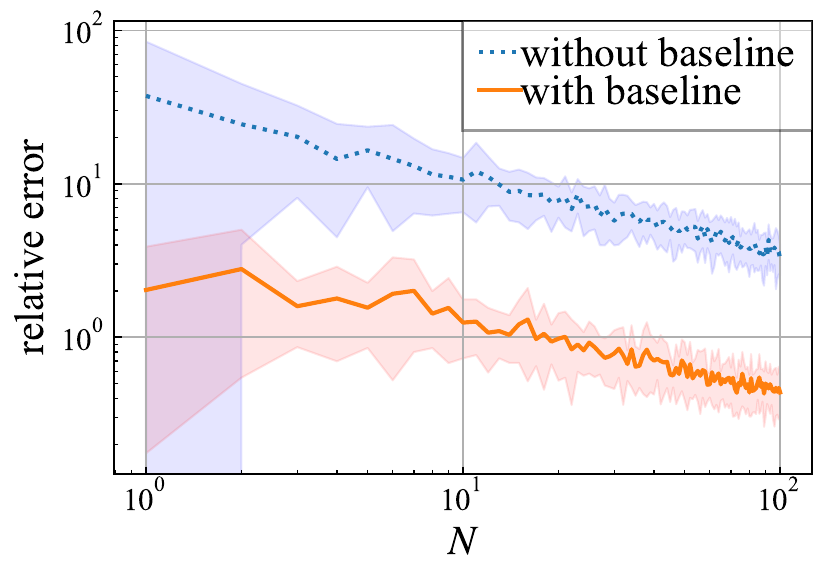}
		\caption{Relative error in gradient estimation.}
		\label{fig:totalerror}
	\end{minipage}
	\begin{minipage}{0.48\columnwidth}
		\centering
		\includegraphics[width=1\columnwidth]{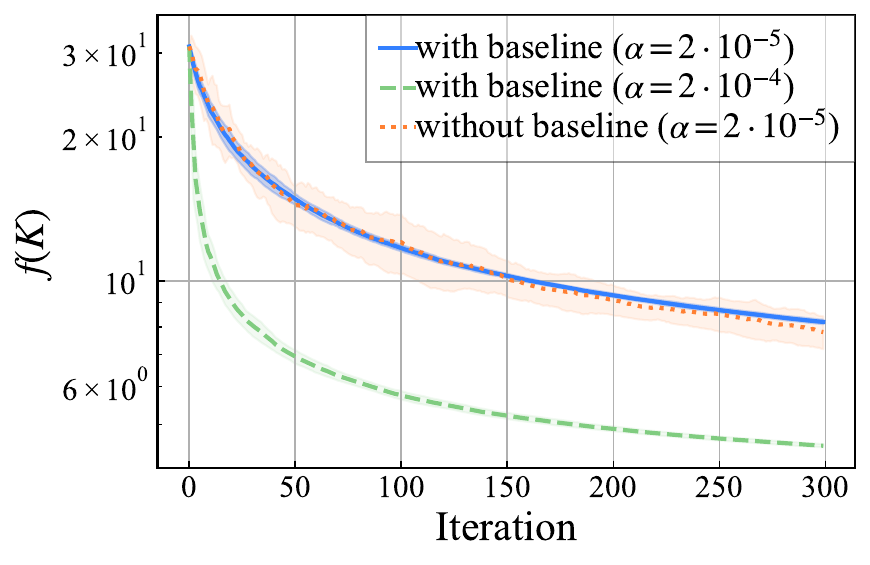}
		\caption{Convergence curve of Algorithm \ref{alg:mfgp}.}
		\label{fig:convergence-curve}
	\end{minipage}
\end{figure}


%% file: conclusion.tex
\section{Conclusion}\label{sec:conclusion}
In this paper, we considered the non-convex optimization problem with convex constraints
based on the output feedback version of LQR problems
under the assumption that system parameters are unknown.
To solve the problem, we proposed the policy gradient algorithm based on the gradient projection method and the zeroth order optimization.
We proved its global convergence to $\varepsilon$-stationary points with high probability.
We also proposed the variance reduction method using the baseline technique and proved that it is almost optimal.
In the numerical experiments, we showed that the baseline technique significantly reduces the variance in the gradient estimation and the model free method can achieve low LQR cost. 

Policy Gradient Projection can be extended to the objective function with regularization terms
using the proximal gradient method. In this setting, we are able to consider trade-offs
between cost function and structure such as sparsity~\cite{park2020structured}.
However, the convergence analysis would be more difficult,
and it is left for a future work.
In addition, the convergence of the gradient method with fixed step size could be slow since the smooth constant $L$ can be large depending on the initial feedback gain $K_0$. To overcome this issue, the gradient method with adaptive step size in the model based setting was considered in~\cite{bu2019lqr} and optimization methods on Riemannian manifolds were studied in~\cite{talebi2022policy}. Therefore, applying the adaptive step size to the model free algorithm is one of the important directions of future works.
Other interesting directions of future works would be
analysis for the natural policy gradient method~\cite{kakade2001natural} or other variants of the policy gradient method.


%% file: appendix.tex
\subsection{Proof of Lemma 1}\label{sec:proof-lemma1}
\begin{proof}
  For $X$, see Lemma 16 in~\cite{mohammadi2021convergence}.

  Let $\mu$ be the largest eigenvalue of $Y$ and
  $v$ be a normalized eigenvector corresponding to $\mu$.
  Note that $\norm{Y}_2 = \mu$ since $Y\succeq 0$.
  In the following, we consider the case $\norm{C v}\geq \xi$ and the case $\norm{C v}\leq \xi$ separately.
  First, we consider the case $\norm{Cv}\geq \xi$.
  Using $f(K) = \tr(YC^\top(Q+K^\top RK)C)$, we have
    $a  \geq \tr(YC^\top(Q+K^\top RK)C)  \geq \tr(YC^\top Q C)$.
  Since $\mu vv^\top \preceq Y$,
    $\tr(YC^\top Q C)  \geq \tr(\mu vv^\top C^\top Q C) = \mu(Cv)^\top Q Cv$.
  Therefore,
    $a  \geq \mu(Cv)^\top Q Cv
       \geq \mu\norm{Cv}^2 \lmin(Q) 
       \geq \mu\xi^2 \lmin(Q)$ and
    $\mu \leq \frac{a}{\xi^2\lmin(Q)}$.
  Next, we consider the case $\norm{Cy}\leq \xi$.
  From \eqref{eq:lyapunovY}, we have
  \begin{align}
    v^\top \qty(A_K Y + Y A_K^\top + \Sigma) v & = \mu v^\top (A_K+A_K^\top)v+v^\top \Sigma v = 0.
  \end{align}
  Since $A_{K_0}+A_{K_0}^\top \leq -2\sigma I$ and $\norm{Cv}\leq \xi$, we obtain
    $v^\top (A_K+A_K^\top)v = v^\top (A_{K_0}+A_{K_0}^\top)v - v^\top (B(K-K_0)C+(B(K-K_0)C)^\top)v \leq  -\sigma$.
  Therefore,
    $-\mu\sigma+v^\top \Sigma v  \geq 0$,
  and thus,
    $\norm{Y}_2 = \mu \leq \frac{\norm{\Sigma}_2}{\sigma}$.
  Combining both cases, we have $\norm{Y}_2 \leq \mathfrak{Y}(a)$.

  For $Y'$, we have
  \begin{align}
       -\frac{\norm{BECY+(BECY)^\top}_2}{\lmin(\Sigma)}\Sigma
       & \preceq -BECY-(BECY)^\top \\
       & \preceq \frac{\norm{BECY+(BECY)^\top}_2}{\lmin(\Sigma)}\Sigma,
  \end{align}
  and Lemma A.2 in~\cite{fatkhullin2021optimizing} yields
  $
    Y^- \preceq Y' \preceq Y^+,
  $
  where 
  \begin{align}
      Y^+ = \frac{\norm{BECY+(BECY)^\top}_2}{\lmin(\Sigma)}Y,\quad Y^- =-Y^+.
  \end{align}
  Therefore,
  \begin{align}
    \norm{Y'}_2  \leq \norm{Y^+}_2 
                 \leq \frac{2\norm{B}_2\norm{C}_2\norm{Y}_2^2}{\lmin(\Sigma)}.\label{eq:norm-Y-prime-mid}
  \end{align}
  Applying the bound on $\norm{Y}_2$ to \eqref{eq:norm-Y-prime-mid}, we have the result.
\end{proof}
\subsection{Proof of Theorem 2}\label{sec:proof-theorem2}
The total error $\normF{\hat \grad f(K) - \grad f(K)}$ can be divided into the bias term $\normF{E\qty[\hat \grad f(K)] - \grad f(K)}$ and the variance term $\normF{\hat \grad f(K) - E\qty[\hat \grad f(K)]}$.

First, we bound the bias term.
The estimated gradient $\hat \grad f(K)$ in~\eqref{eq:hatgrad} can be expressed in the form
\begin{align}
  \hat \grad f(K) & = \frac{1}{Nr}\sum_{i=1}^N \tilde f_\tau(K+rU_i;x_i(0))U_i.\label{eq:hatgrad-ftau}
\end{align}

For any initial state $x(0)$ and $r > 0$, we define 
the smooth function $g_r(K)$ by
$
  g_r(K) := E_{U\sim \mathcal{B}}\qty[f(K+rU)],
$
where $\mathcal{B}$ is the uniform distribution over the set $\{U\in \R^{m\times p}\mid \norm{U}\leq \sqrt{mp}\}$.
Then, the bias in $\hat \grad f(K)$ can be divided into two parts as follows.
\begin{align}
  &\normF{\grad f(K)-E\qty[\hat\grad f(K)]}\leq \normF{\grad f(K)-\grad g_r(K)} \\
  &+ \normF{\grad g_r(K)-E\qty[\hat\grad f(K)]}, \label{eq:bias}
\end{align}
where the expectation is taken over $x_i(0)\sim \mathcal D$ and $U_i \sim \mathcal S$.

For the first term in \eqref{eq:bias}, we have the following bound.
\begin{lemma}\label{lem:biasr}
  For any $K\in S(a)$ and $r \leq r_0$,
  $
    \normF{\grad f(K)-\grad g_r(K)} \leq Lr\sqrt{mp},
  $
  where $L$ is the Lipschitz constant defined as~\eqref{eq:L} of $\grad f$ on $S(2a)$.
\end{lemma}
\begin{proof}
  From $L$-smoothness of $f$, we have $\normF{\grad f(K) - \grad f(K+rU)}\leq Lr\sqrt{mp}$. Therefore, $\normF{\grad f(K)-\grad g_r(K)} = E\qty[\normF{\grad f(K)-\grad f(K+rU)}]\leq Lr\sqrt{mp}$. \qed
\end{proof}

Lemma 26 in Supplementary material of \cite{fazel2018global} implies
  $\grad g_r(K)  = E\qty[\frac{1}{r}\tilde f(K+rU;x_0)U]$,
and \eqref{eq:hatgrad-ftau} yields
 $E\qty[\hat \grad f(K)]  = E\qty[\frac{1}{r} \tilde f_\tau(K+rU;x_0)U]$,
where the expectation is taken over $x_0\sim \mathcal D$ and $U\sim \mathcal S$.
By using these relations, we have the following upper bound of the second term in~\eqref{eq:bias}.
\begin{align}
         &\normF{\grad g_r(K)- E\qty[\hat\grad f(K)]}\\
         &\leq \frac{1}{rN}\sum_{i=1}^N E\qty[\abs{\tilde f(K+rU_i;x_i)-\tilde f_\tau(K+rU_i;x_i)}\normF{U_i}].\label{eq:biastau}
\end{align}
To bound the right-hand side, we introduce the following lemma.
\begin{lemma}\label{lem:norm-x-t}
  For $K\in S(a)$ with $a\in \R$ and $x(t)$, which follows \eqref{eq:closedloop}, we have
  \begin{align}
    \norm{x(t)} & \leq \frac{2\mathfrak{Y}(a)\mathfrak{A}(a)}{\lmin(\Sigma)}e^{-(\lmin(\Sigma)/\mathfrak{Y}(a))t}\norm{x(0)}^2,\label{eq:norm-x-t}
  \end{align}
  where 
  $
    \mathfrak{A}(a) = \norm{A}_2+\norm{B}_2\norm{C}_2\kappa(a).
  $
\end{lemma}
\begin{proof}
  From Lemma 12 in \cite{mohammadi2021convergence}, we have
  $
    \norm{e^{At}}_2^2 \leq \qty(\norm{Y}_2/\lmin(Y))e^{-(\lmin(\Sigma)/\norm{Y}_2)t}.
  $
  Therefore,
  \begin{align}
    \norm{x(t)}^2 
                   \leq \qty(\norm{Y}_2/\lmin(Y))e^{-(\lmin(\Sigma)/\norm{Y}_2)t}\norm{x(0)}^2.\label{eq:lem12}
  \end{align}
  Lemma \ref{lem:xybound} yields $\norm{Y}_2  \leq \mathfrak{Y}(a)$.
  This and Lemma A.5 in~\cite{fatkhullin2021optimizing} imply
  \begin{align}
    \lmin(Y) &\geq \frac{\lmin(\Sigma)}{2\norm{A_K}_2} 
    \geq \frac{\lmin(\Sigma)}{2(\norm{A}_2+\norm{B}_2\norm{C}_2\kappa(a))}.
  \end{align}
  Substituting these inequalities into \eqref{eq:lem12}, we obtain \eqref{eq:norm-x-t}.\qed
\end{proof}

We are in a position to obtain an upper bound on the left side of~\eqref{eq:biastau}.
\begin{lemma}
  \label{lem:biastau}
  For any $\tau \geq 0$, $r \leq r_0$ and $K\in S(a)$,
  \begin{align}
    \normF{\grad g_r(K)-E\qty[\hat\grad f(K)]} =
    O(e^{-\eta \tau}/r),\label{eq:biastau2}
  \end{align}
  where $\eta = \lmin(\Sigma)/\mathfrak{Y}(2a)$.
\end{lemma}
\begin{proof}
  Lemma~\ref{lem:smallr} implies that $K+rU_i \in S(2a)$, and we have
  \begin{align}
    \normF{\grad g_r(K)-E\qty[\hat\grad f(K)]}  &\leq  \frac{1}{r} E\qty[x(\tau)^\top Xx(\tau)\normF{U}]   \\
    &\leq \frac{\sqrt{mp}}{r} \norm{X}_2E\qty[\norm{x(\tau)}^2],
  \end{align}
  where
  $x(t)$ follows
    $\dot x(t)  = A_{K+rU} x(t)$, and
  we used the fact
  \begin{align}
    \tilde f(K+rU;x(0))-\tilde f_\tau(K+rU;x(0)) & = x(\tau)^\top Xx(\tau),
  \end{align}
  where $X$ is the solution to~\eqref{eq:lyapunovX}.
  From Lemma \ref{lem:norm-x-t}, we have
  \begin{align}
    \norm{x(\tau)}^2 & \leq \frac{2\mathfrak{Y}(2a)\mathfrak{A}(2a)}{\lmin(\Sigma)}e^{-(\lmin(\Sigma)/\mathfrak{Y}(2a))\tau}\norm{x(0)}^2.
  \end{align}
  Thus, \eqref{eq:biastau2} holds,
  because $E[\norm{x(0)}^2]=\tr(\Sigma)$.\qed
\end{proof}

Next, we obtain an upper bound of the variance term.
\begin{lemma}
  \label{lem:variance}
  For any $\varepsilon > 0, \delta > 0$ and $K\in S(a)$, if $N = O((\log 1/\delta)/\varepsilon^4)$, we have
  \begin{align}
    \Probability\qty(\normF{\hat \grad f(K)-E\qty[\hat \grad f(K)]}\geq \varepsilon)\leq \delta
  \end{align}
\end{lemma}
\begin{proof}
    Using matrix Bernstein inequality~\cite{tropp2012user}, we obtain the result in the same way with Lemma 27 in Supplementary material of~\cite{fazel2018global}. \qed
\end{proof}

Combining Lemma~\ref{lem:biasr},~\ref{lem:biastau}, and~\ref{lem:variance} completes the proof of Theorem~\ref{thm:totalerror}.

\subsection{Proof of Theorem~\ref{thm:estimate-f}}\label{proof:thm:estimate-f}
For any $t \geq 0$, the observation $\bar y(t;x(0))$ is determined by
    $\bar y(t;x(0)) = Fx(t)$
with $F := [C \quad Ce^{(A-BKC) h_1} \cdots Ce^{(A-BKC) h_{D-1}}]^\top$. 
Conversely, $x(t)$ is determined by $\bar y(t;x(0))$ if $D$ is large enough.
\begin{lemma}\label{lem:reconstruction}
  Let $\beta = 2(\norm{A}_2+\norm{B}_2\norm{C}_2\kappa(a))$.
  For any $T>0$, if $D>2(n-1)+\frac{T}{2\pi}\beta$, $F$ is column full rank and
    $x(t)  = F^+\bar y(t;x(0))$
  with $F^+ = (F^\top F)^{-1}F^\top$. 
\end{lemma}
\begin{proof}
  Let $\beta' = \max_{i, j}(\Im(\lambda_i(A_K)-\lambda_j(A_K)))$.
   Theorem 2 in \cite{modares2016optimal}
   and the assumption that $(A, C)$ is observable
   imply that if $D>2(n-1)+\frac{T}{2\pi}\beta'$,
   $F$ is column full rank.
  Therefore, $F^+$ is well-defined and
  \begin{align}
      (F^\top F)^{-1}F^\top\bar y(t;x(0)) &= (F^\top F)^{-1}F^\top F x(t) = x(t).
  \end{align}
  Thus, it is sufficient to show $\beta \geq \beta'$.
  We have
  \begin{align}
      \beta' \leq 2\max_{i}(|\lambda_i(A_K)|)\leq 2\norm{A_K}_2\leq \beta,
  \end{align}
  which completes the proof.\qed
\end{proof}

As a corollary, we can show $\tilde f(K;x(t))$ can be expressed as a quadratic form in terms of $\bar y(t;x(0))$.
\begin{corollary}\label{cor:f-from-y-bar}
  For any $x(0)$,
  \begin{align}
    \tilde f(K;x(t)) & = \bar y(t;x(0))^\top P(K)\bar y(t;x(0)),\label{eq:f-from-y-bar}
  \end{align}
  where $P(K) = (F^+)^\top X F^+$.
\end{corollary}

According to \cite{modares2016optimal}, the matrix $P(K)$ satisfies the Bellman equations \eqref{eq:bellmaneq}
for any $s \geq 0$ and $\qty{x_j(t)}_{j=1, \dots, \frac{n(n+1)}{2}}$, which follow system~\eqref{eq:system}.

\begin{lemma} \label{lem:f-from-y-bar}
  Assume that $x_j(0)x_j(0)^\top - x_j(s)x_j(s)^\top(j=1, \dots, \frac{n(n+1)}{2})$
  are linearly independent on $\mathbb{S}^n$.
  For any solution $\hat P(K)$ of~\eqref{eq:bellmaneq}, initial condition $x(0)$, and $t \geq 0$,
  we have
  \begin{align}
    \tilde f(K;x(t)) = \bar y(t;x(0))^\top \hat P(K)\bar y(t;x(0)).
  \end{align}
\end{lemma}
\begin{proof}
  Let $v_j = x_j(0)x_j(0)^\top - x_j(s)x_j(s)^\top$
  and $w_j = \bar{y}(0;x_j(0))\bar{y}(0;x_j(0))^\top - \bar{y}(s;x_j(0)) \bar{y}(s;x_j(0))^\top$.
  We define $V$ by the linear space generated by $\{v_j\}_j$ and $W$ by the linear space generated by $\{w_j\}_j$.
  Since $V \subset \mathbb{S}^n$ and $\dim V = \dim \mathbb{S}^n = \frac{n(n+1)}{2}$, we have $V = \mathbb{S}^n$.
  Thus, the set $\{v_i\}_i$ is a basis of $\mathbb{S}^n$.
  Let $v = x(t)x(t)^\top\in \mathbb{S}^n$ for $x(0)\in \R^n$.
  Then, there exists the sequence $\{\alpha_i\}_i$ such that
    $v  = \sum_{i=1}^{\frac{n(n+1)}{2}} \alpha_i v_i$.
  Define the linear map $\mathcal{F}:V\to W$ by
    $\mathcal{F}(v')  = Fv'F^\top$,
  where $v'\in V$.
  Note that $\mathcal{F}(v_j) = w_j$.
  From~\eqref{eq:bellmaneq}, we have
    $\q<\hat P(K), \mathcal F(v_j)> = \q<P(K), \mathcal F(v_j)>$, and thus
    $\q<\hat P(K), \mathcal F(v)> =\q<P(K), \mathcal F(v)>$.
  Then, Eq.~\eqref{eq:f-from-y-bar} yields
  \begin{align}
    \q<P(K), \mathcal F(v)> & = \bar y(t;x(0))^\top P(K)\bar y(t;x(0))  = \tilde f(K;x(t)).
  \end{align}
  Therefore,
  \begin{align}
    \bar y(t;x_0)^\top \hat P(K)\bar y(t;x_0) &= \q<\hat P(K), \mathcal F(v)>
    = \tilde f(K;x(t)),
  \end{align}
  which completes the proof.\qed
\end{proof}

Since $\tilde f_\tau(K;x(0)) = \tilde f(K;x(0)) - \tilde f(K;x(\tau))$, Corollary \ref{cor:f-from-y-bar} and Lemma \ref{lem:f-from-y-bar} ensure that Theorem \ref{thm:estimate-f} holds.